\documentclass[a4paper,11pt]{article}
\usepackage{amsmath, amsfonts, amscd, amssymb, amsthm, enumerate, xypic}

\def\Mat{\text{M}}
\def\GL{\text{GL}}
\def\bfB{\mathbf{B}}

\def\id{\text{id}}
\def\card{\#\,}
\newcommand{\Ker}{\operatorname{Ker}}
\newcommand{\Vect}{\operatorname{span}}
\newcommand{\im}{\operatorname{Im}}
\newcommand{\tr}{\operatorname{tr}}
\newcommand{\Sp}{\operatorname{Sp}}
\newcommand{\rk}{\operatorname{rk}}

\renewcommand{\setminus}{\smallsetminus}

\def\F{\mathbb{F}}
\def\K{\mathbb{K}}

\def\calA{\mathcal{A}}

\def\calF{\mathcal{F}}
\def\calG{\mathcal{G}}

\def\calP{\mathcal{P}}
\def\calQ{\mathcal{Q}}

\def\calV{\mathcal{V}}

\def\lcro{\mathopen{[\![}}
\def\rcro{\mathclose{]\!]}}

\theoremstyle{definition}
\newtheorem{Def}{Definition}
\newtheorem{Not}[Def]{Notation}

\theoremstyle{plain}
\newtheorem{theo}{Theorem}
\newtheorem{prop}[theo]{Proposition}
\newtheorem{cor}[theo]{Corollary}
\newtheorem{lemme}[theo]{Lemma}

\theoremstyle{plain}

\theoremstyle{remark}
\newtheorem{Rems}{Remarks}
\newtheorem{Rem}[Rems]{Remark}

\title{The affine preservers of non-singular matrices}
\author{Cl\'ement de Seguins Pazzis\footnote{Professor of Mathematics at Lyc\'ee Priv\'e Sainte-Genevi\`eve, 2, rue
de l'\'Ecole des Postes, 78029 Versailles Cedex, FRANCE.}
\footnote{e-mail address: dsp.prof@gmail.com}}

\begin{document}

\thispagestyle{plain}
\maketitle

\begin{abstract}
When $\K$ is an arbitrary field, we study the affine automorphisms of $\Mat_n(\K)$ that stabilize $\GL_n(\K)$.
Using a theorem of Dieudonn\'e on maximal affine subspaces of singular matrices,
this is easily reduced to the known case of linear preservers when $n>2$ or $\card \K>2$. We include a short new proof of
the more general Flanders' theorem for affine subspaces of $\Mat_{p,q}(\K)$ with bounded rank.
We also find that the group of affine transformations of $\Mat_2(\F_2)$ that stabilize $\GL_2(\F_2)$
does not consist solely of linear maps. Using the theory of quadratic forms over $\F_2$, we construct explicit
isomorphisms between it, the symplectic group $\Sp_4(\F_2)$ and the symmetric group $\frak{S}_6$.
\end{abstract}



\section{Introduction}

Here, $\K$ will denote an arbitrary field and $n$ a positive integer.
By an affine transformation of an affine space, we will always mean an affine \emph{bijective} map.
We let $\Mat_{n,p}(\K)$ denote the set of matrices with $n$ rows, $p$ columns and entries in $\K$, and
$\GL_n(\K)$ the set of non-singular matrices in the algebra $\Mat_n(\K)$ of square matrices of order $n$. \\
For $(i,j)\in \lcro 1,n\rcro \times \lcro 1,p\rcro$, we let $E_{i,j}$ denote the elementary matrix
of $\Mat_{n,p}(\K)$ with entry $1$ at the $(i,j)$ spot and zero elsewhere.

\noindent We make the group $\GL_n(\K) \times \GL_p(\K)$ act on the set of linear subspaces of $\Mat_{n,p}(\K)$
by
$$(P,Q).V:=P\,V\,Q^{-1}.$$
Two linear subspaces of the same orbit will be called \textbf{equivalent}
(this means that they represent the same set of linear transformations from a $p$-dimensional vector space
to an $n$-dimensional vector space).

\noindent For non-singular matrices $P$ and $Q$ in $\GL_n(\K)$, we define
$$u_{P,Q} : \begin{cases}
\Mat_n(\K) & \longrightarrow \Mat_n(\K) \\
M & \longmapsto P\,M\,Q
\end{cases}\quad \text{and} \quad v_{P,Q} :
\begin{cases}
\Mat_n(\K) & \longrightarrow \Mat_n(\K) \\
M & \longmapsto P\,M^t\,Q.
\end{cases}$$
Clearly, these are non-singular endomorphisms of the vector space $\Mat_n(\K)$
which map $\GL_n(\K)$ onto itself, and the subset
$$\calG_n(\K):=\bigl\{u_{P,Q} \mid (P,Q)\in \GL_n(\K)^2\bigr\} \cup \bigl\{v_{P,Q} \mid (P,Q)\in \GL_n(\K)^2\bigr\}$$
is a subgroup of $\GL(\Mat_n(\K))$, which we will call the \emph{Frobenius group.}

\vskip 2mm
\noindent One of the earliest result on linear preservers problems is the following one of Dieudonn\'e
\cite{Dieudonne} following a classical work of Frobenius \cite{Frobenius}:
\begin{theo}\label{Dieudo1}
The group $\calG_n(\K)$ consists of all the automorphisms of the vector space $\Mat_n(\K)$ which stabilize $\GL_n(\K)$.
\end{theo}

\noindent To prove this, Dieudonn\'e established a major result on \emph{singular} subspaces of $\Mat_n(\K)$ (i.e. linear or affine subspaces
which contain only singular matrices):

\begin{theo}\label{Dieudo2}
Let $\calV$ be a singular affine subspace of $\Mat_n(\K)$. Then $\dim \calV \leq n(n-1)$. \\
If $\dim \calV =n(n-1)$, then either $\calV$ or $\calV^t$ is equivalent to
the linear subspace $\Bigl\{\begin{bmatrix}
M & 0
\end{bmatrix} \mid M \in \Mat_{n-1}(\K)\Bigr\}$ unless $n=2$, $\card \K=2$ and
$\calV$ is not a linear subspace of $\Mat_n(\K)$.
\end{theo}

\noindent Notice that all matrices in the plane $\biggl\{\begin{bmatrix}
x & y \\
0 & x+1
\end{bmatrix}\mid (x,y)\in \F_2^2\biggr\}$ are singular, hence the exceptional case mentioned in the theorem.

\vskip 2mm
\noindent Dieudonn\'e initially restricted his study to linear transformations because he wanted
to generalize the classical description for the orthogonal group of the quadratic space $(\Mat_2(\K),\det)$.
As we shall see, the generalization to affine transformations is extremely easy except for the exceptional case of $n=2$ and $\K \simeq \F_2$,
but it is very intriguing that Dieudonn\'e, at the time the best specialist in the theory
of classical groups, completely overlooked the connection between this exceptional case, the symplectic group $\Sp_4(\F_2)$
and the theory of quadratic forms over $\F_2$. It is our main goal to fill this bizarre hole in Dieudonn\'e's celebrated paper.

\vskip 2mm
\noindent In Section \ref{aff}, we will quickly determine the affine transformations of $\Mat_n(\K)$ which stabilize $\GL_n(\K)$
when $n>2$ of $\card \K>2$, and show briefly that we cannot expect to obtain results of the type of \cite{dSPlinpresGL}
for singular affine endomorphisms of $\Mat_n(\K)$. Our (very short) proof will involve
Theorems \ref{Dieudo1} and \ref{Dieudo2}. We will use this opportunity
to give a brand new proof of Theorem \ref{Dieudo2} in the more general formulation of Flanders
with no restriction on the field (except for the case where an explicit counter-example exists).
This is, to our knowledge, the shortest proof of this theorem, appealing only to basic linear algebra.
In the last section, we will investigate the case $n=2$ and $\K \simeq \F_2$, which will involve
the theory of quadratic forms over $\F_2$ (the reader will find proofs of the more basic statements in section 10.4 of \cite{Scharlau}).

\noindent Let us finish this introductory section by stating our main theorem:

\begin{theo}\label{affpresdSP}
Let $n$ be an integer and $\K$ an arbitrary field. Let $\calA\calG_n(\K)$ denote the group of
affine automorphisms of $\Mat_n(\K)$ which stabilize $\GL_n(\K)$.
\begin{enumerate}[(a)]
\item If $n>2$ or $\card \K>2$, then all affine automorphisms of $\Mat_n(\K)$ which stabilize
$\GL_n(\K)$ are linear, hence $\calA\calG_n(\K)=\calG_n(\K)$.
\item Not all elements of $\calA\calG_2(\F_2)$ are linear maps.
\item The natural action of $\calA\calG_2(\F_2)$
on $\GL_2(\F_2)$ induces a group isomorphism from $\calA\calG_2(\F_2)$ to the symmetric group $\frak{S}\bigl(\GL_2(\F_2)\bigr)$.
Assigning its linear part to every transformation in $\calA\calG_2(\F_2)$
induces a group isomorphism from $\calA\calG_2(\F_2)$ to the symplectic group of the form
$(A,B) \mapsto \det(A+B)-\det(A)-\det(B)$.
\end{enumerate}
\end{theo}

\section{The case $n>2$ or $\card \K>2$}\label{aff}

We start by a quick proof of statement (a) in Theorem \ref{affpresdSP}.
Let $u$ be an affine automorphism of $\Mat_n(\K)$ such that $u(P)$ is non-singular for every $P \in \GL_n(\K)$.
We assume $n>2$ or $\card \K>2$.
For $i \in \lcro 1,n\rcro$, let $V_i$ denote the set of all matrices of $\Mat_n(\K)$ with a zero $i$-th column.
Then $V_i$ is a $(n^2-n)$-dimensional affine subspace of $\Mat_n(\K)$ consisting only of non-singular matrices.
It follows that the same is true of $u^{-1}(W_i)$. By Theorem \ref{Dieudo2}, we deduce that
$u^{-1}(W_i)$ is a linear subspace of $\Mat_n(\K)$. Consequently,
$$u^{-1}\{0\}=\underset{i=1}{\overset{n}{\bigcap}}\,u^{-1}(W_i)$$
is a linear subspace of $W$, hence it contains $0$, which proves $u$ is linear. Using
Theorem \ref{Dieudo1}, we conclude that $u \in \calG_n(\K)$, which essentially finishes the proof of statement (a).

\vskip 3mm
\noindent Let us now give a corollary for this part of Theorem \ref{affpresdSP}:

\begin{cor}
Let $n$ be an integer. Assume $n \neq 2$ or $\card \K>2$. Then:
\begin{enumerate}[(i)]
\item The group $\calG_n(\K)$ consists of all the affine endomorphisms $u$ of $\Mat_n(\K)$ such that
$u^{-1}(\GL_n(\K))=\GL_n(\K)$.
\item If $n>1$ or $\card \K>2$, then $\calG_n(\K)$ consists of all the affine endomorphisms $u$ of $\Mat_n(\K)$ such that
$u(\GL_n(\K))=\GL_n(\K)$.
\end{enumerate}
\end{cor}

\noindent Notice that $u : \lambda \mapsto 1$ is a non-linear affine endomorphism of $\Mat_1(\F_2)$ such that
$u(\GL_1(\F_2))=\GL_1(\F_2)$.

\begin{proof}
Statement (i) derives from Theorem \ref{affpresdSP} in the very same way that
statement (ii) in Theorem 1 of \cite{dSPlinpresGL} derived from statement (iii) of that same theorem
(i.e. we show that the kernel of the linear part of $u$ is trivial). \\
Assume now $n>1$ or $\card \K>2$, and let $u$ be an affine endomorphism of $\Mat_n(\K)$ such that $u(\GL_n(\K))=\GL_n(\K)$.
It then suffices to show that $u$ is onto, which comes from the following lemma.
\end{proof}

\begin{lemme}\label{GLaffengendre}
Assume $n>1$ or $\card \K>2$. Then no strict affine subspace of $\Mat_n(\K)$ may contain $\GL_n(\K)$.
\end{lemme}

\begin{proof}
The case $n=1$ is trivial so we assume $n \geq 2$.
Assume there is a linear hyperplane of $\Mat_n(\K)$ which contains $\GL_n(\K)$. Then there would
be a non-zero matrix $A \in \Mat_n(\K)$ and an integer $a$ such that $\forall P \in \GL_n(\K), \; \tr(AP)=a$. \\
It would follow that $\tr(QAP)=\tr(APQ)=a$ for every $(P,Q)\in \GL_n(\K)^2$, hence every matrix with rank $r:=\rk A$
would have trace $a$. If $r \geq 2$, the two rank $r$ matrices
$\sum_{i=1}^r E_{i,i}$ and
$E_{1,2}+E_{2,1}+\sum_{i=2}^r E_{i,i}$ have different traces.
If $r=1$, then the rank $1$ matrices $E_{1,1}$ and $E_{1,2}$ have different traces. In any case, we have a contradiction.
\end{proof}

\noindent The reader should not expect any neat description of the singular affine endomorphisms stabilizing $\GL_n(\K)$
(unlike our results on linear endomorphisms, cf.\ \cite{dSPlinpresGL}).
It is indeed very easy to build large affine subspaces consisting solely of non-singular matrices, a good example
being the subspace $I_n+T_n^{++}(\K)$, where $T_n^{++}(\K)$ denotes the subset
of strictly upper triangular matrices in $\Mat_n(\K)$ (more generally, we can take a linear subspace $V$ of nilpotent matrices
and consider the affine subspace $I_n+V$). Any affine map
$u : \Mat_n(\K) \rightarrow I_n+T_n^{++}(\K)$ then stabilizes $\GL_n(\K)$ !

\section{A Flanders theorem for affine subspaces}

This section is devoted to a generalization of Theorem \ref{Dieudo2}.
Prior to this, such a generalization had been proven by Flanders \cite{Flanders} for linear subspaces with a large enough field,
and Meshulam \cite{Meshulam} for linear subspaces with an arbitrary field. Our starting point will resemble a lot to that of Flanders,
but we will be very careful to avoid multiplying by scalars as much as possible.

\begin{Not}
For $p \in \lcro 0,n\rcro$, we set $C_p:=\Bigl\{\begin{bmatrix}
M & 0
\end{bmatrix} \mid M \in \Mat_{n,p}(\K)\Bigr\}\subset \Mat_n(\K)$.
\end{Not}

\begin{theo}\label{Flanders} ${}$ \\
Let $\calV$ be an affine subspace of $\Mat_n(\K)$, and set $r:=\max\{\rk A \mid A \in \calF\}$. \\
Then $\dim \calV \leq nr$. If in addition $\dim \calV=nr$, then
$\calV$ is equivalent to either $C_r$ or ${}^tC_r$ except when $n=2$, $\card \K=2$, $r=1$
and $0\not\in \calV$.
\end{theo}

\noindent The proof will involve the following lemma, which we will prove right away:

\begin{lemme}\label{splitlemma}
Let $H$ be a $pq$-dimensional linear subspace of $\Mat_{p,q}(\K) \times \Mat_{q,p}(\K)$ with:
\begin{enumerate}[(i)]
\item $\forall \bigl((L,C),(L',C')\bigr)\in H^2, \; \forall P \in \GL_q(\K), \; LPC'+L'PC=0$;
\item if $\card \K>2$ and $p=q=1$, then $\forall (L,C)\in H, \; \forall P \in \GL_q(\K), \; LPC=0$.
\end{enumerate}
Then either $H=\Mat_{p,q}(\K) \times \{0\}$ or $H=\{0\} \times \Mat_{q,p}(\K)$, or $p=q=1$ and $\K \simeq \F_2$.
\end{lemme}

\begin{proof}
It suffices to show that $H \subset \Mat_{p,q}(\K) \times \{0\}$ or $H \subset \{0\} \times \Mat_{q,p}(\K)$
except in the exceptional case mentioned above.
We will make great use of the following easy fact:
for any $C \in \Mat_{q,p}(\K)\setminus \{0\}$, one has $\sum_{P \in \GL_q(\K)} \im PC=\K^q$. \\
Assume $(L,C) \mapsto L$ is not one-to-one from $H$ and choose a non-zero $(0,C)$ in its kernel.
Let $(L',C') \in H$. Then (i) shows $\forall P \in \GL_q(\K), \; L'PC=0$ hence $L'=0$ by the preliminary remark, and
$(L',C') \mapsto C'$ is then one-to-one from $H$. In any case, replacing $H$ with $\bigl\{(B^t,A^t)\mid (A,B)\in H\bigr\}$ shows we can assume
$(L,C) \mapsto L$ is one-to-one on $H$, hence a linear isomorphism from $H$ to $\Mat_{p,q}(\K)$,
which entails that $H=\bigl\{(L,\alpha(L))\mid L \in \Mat_{p,q}(\K)\bigr\}$ for some linear map $\alpha$.
\vskip 2mm
\noindent Assume $\card \K>2$ and $p=q=1$. Let $L \in \Mat_{p,q}(\K) \setminus \{0\}$. Then $L\,P\,\alpha(L)=0$ for every $P \in \GL_q(\K)$, hence
$\alpha(L)=0$ using the preliminary remark, and we deduce that $H \subset \Mat_{p,q}(\K) \times \{0\}$.
\vskip 2mm
\noindent Assume now $(p,q) \neq (1,1)$ and $\alpha$ is non-zero. Assume there exists $L \in \Mat_{p,q}(\K) \setminus \{0\}$
such that $\alpha(L)=0$. For every $L' \in \Mat_{p,q}(\K)$, we then have $\forall P \in \GL_n(\K), \; L\,P\,\alpha(L')=0$
hence $\alpha(L')=0$ (using the preliminary remark again). This contradicts a previous assumption so $\alpha$ is actually one-to-one. \\
Let $L \in \Mat_{p,q}(\K) \setminus \{0\}$. Then
$\alpha(L) \neq 0$, so the identity $\forall P \in \GL_q(\K), \; L'\,P\,\alpha(L)=-L\,P\,\alpha(L')$
and the preliminary remark entail that $\im L' \subset \im L$ for every $L' \in \Mat_{p,q}(\K)$.
This shows $p=1$ since $L$ can be chosen with rank $1$. Then $q>1$, we can choose linearly independent $L$ and $L'$ in $\Mat_{1,q}(\K)$,
so $\alpha(L)$ and $\alpha(L')$ are linearly independent;
we can also choose $Y_1 \in \Ker L \setminus \{0\}$ and $Y_2\in \K^q \setminus (\Ker L \cup \Ker L')$, and there
exists then some $P \in \GL_q(\K)$ such that $P\,\alpha(L')=Y_1$ and $P\,\alpha(L)=Y_2$.
This yields $L\,P\,\alpha(L')+L'\,P\,\alpha(L)=L'Y_2 \neq 0$, contradicting (i).
\end{proof}

\begin{proof}[Proof of Theorem \ref{Flanders}]
The case $n=2$, $\K \simeq \F_2$ and $\calV$ is a linear subspace is straightforward with the help of the following easy lemma:
given two rank $1$ matrices $M$ and $N$ of $\Mat_n(\K)$, if $M+N$ has rank $1$, then $\Ker M=\Ker N$ or $\im M=\im N$. \\
We now discard the case $n=2$ and $\card \K=2$, and let $V$ denote the translation vector space of $\calV$. \\
Let $A=\begin{bmatrix}
P_1 & C_1 \\
L_1 & \alpha_1
\end{bmatrix} \in \calV$, with blocks $P_1$, $L_1$, $C_1$ and $\alpha_1$ respectively of size
$(r,r)$, $(n-r,r)$, $(r,n-r)$ and $(n-r,n-r)$
(in the rest of the proof, all the block decompositions will have the same configuration). We assume $\rk P_1=r$. \\
Gaussian elimination shows $A$ is equivalent to
$\begin{bmatrix}
P_1 & C_1 \\
0 & \alpha_1-L_1\,P_1^{-1}\, C_1
\end{bmatrix}$, hence
\begin{equation}\label{fundit}
L_1\,P_1^{-1}\, C_1-\alpha_1=0.
\end{equation}
Write every $M \in V$ as
$M=\begin{bmatrix}
K(M) & C(M) \\
L(M) & \alpha(M)
\end{bmatrix}$, set $W:=\Ker K$ i.e. $W$ is the linear subspace of matrices of $V$ having the form
$\begin{bmatrix}
0 & ? \\
? & ?
\end{bmatrix}$. Set also $E:=\Mat_{n-r,r}(\K) \times \Mat_{r,n-r}(\K)$ and $\varphi(M):=(L(M),C(M))\in E$ for any $M \in W$. \\
For every $M \in W$, the matrix $A+M$ belongs to $\calV$ and has $P_1$ as left-upper block, hence \eqref{fundit} shows:
$$(L(M)+L_1)\,P_1^{-1}\,(C(M)+C_1)-\alpha_1-\alpha(M)=0.$$
Subtracting \eqref{fundit}, we deduce:
\begin{equation}\label{general}
\forall M \in W, \; L(M)\,P_1^{-1}C(M)=\alpha(M)-L(M) \,P_1^{-1}\, C_1-L_1\, P_1^{-1}\,C(M).
\end{equation}
Notice that the left-hand side of the equality is a quadratic form $q$ of $M$ on $V$, whilst the right-hand side is a linear form.
We deduce\footnote{For (ii), compute the polar form of $q$ as defined by $b_q(M,N):=q(M+N)-q(M)-q(N)$.}:
\begin{enumerate}[(i)]
\item if $\card K>2$, then $\forall M \in W, \; L(M)\,P_1^{-1}\,C(M)=0$;
\item in any case, $\forall (M,N) \in W^2, \; L(M)\,P_1^{-1}\,C(N)+L(N)\,P_1^{-1}\,C(M)=0$.
\end{enumerate}
Using an equivalence, we lose no generality assuming that $\calV$ contains $J_r:=\begin{bmatrix}
I_r & 0 \\
0 & 0
\end{bmatrix}$. \\
Applying \eqref{general} to $A=J_r$ shows $\varphi$ is one-to-one\footnote{Indeed, $\bigl(L(M),C(M)\bigr)=(0,0)$ entails $\alpha(M)=L(M)\,C(M)=0$.}
and $\varphi(W)$ is a totally singular subspace for the non-degenerate symmetric bilinear form
$b : \bigl((L,C),(L',C')\bigr) \mapsto \tr(LC'+L'C)$ on $E$. Hence $\rk \varphi \leq (\dim E)/2=r\,(n-r)$ and the rank theorem shows:
$$\dim V =\rk K+\dim W= \rk K+\rk \varphi \leq r^2+r\,(n-r)=n\,r.$$
Assume now $\dim V=n\,r$. Then $K$ is onto and $\varphi(W)$ has dimension $r\,(n-r)$.
For every $P \in \GL_r(\K)$, we can then find some $A\in \calV$ with $P^{-1}$ as first block.
Lemma \ref{splitlemma} thus applies to $\varphi(W)$. By transposing $\calV$ if necessary, we lose no generality assuming
$\varphi(W)=\Mat_{n-r,r}(\K) \times \{0\}$, in which case\footnote{Using again $L(M)\,C(M)=\alpha(M)$ for every $M \in V$.}
$W$ consists of all matrices of the form $\begin{bmatrix}
0 & 0 \\
? & 0
\end{bmatrix}$. The factorization lemma for affine maps helps us then recover affine maps
$\tilde{C}$ and $\tilde{\alpha}$ such that
$$\calV=\Biggl\{
\begin{bmatrix}
A & \tilde{C}(A) \\
B & \tilde{\alpha}(A)
\end{bmatrix} \mid A \in \Mat_r(\K), \;B\in \Mat_{n-r,r}(\K)\Biggr\}$$
and we have seen that $\tilde{C}(0)=0$ and $\tilde{\alpha}(0)=0$, hence $\tilde{C}$ and $\tilde{\alpha}$ are linear! \\
Let $P \in \GL_r(\K)$. Then identity \eqref{fundit} shows $\forall B \in \Mat_{n-r,r}(\K), \; BP^{-1} \tilde{C}(P)=\alpha(P)$.
Taking $B=0$ shows $\alpha(P)=0$, then taking all possible $B$'s shows $\tilde{C}(P)=0$.
The classical result $\Vect \bigl(\GL_r(\K)\bigr)=\Mat_r(\K)$ (use Lemma \ref{GLaffengendre} or more simply Lemma 3 of \cite{dSPlinpresGL})
then entails $\tilde{\alpha}=0$, $\tilde{C}=0$ and~$\calV=C_r$.
\end{proof}

\noindent The next corollary will easily follow with a standard line of reasoning.

\begin{cor}
Given positive integers $n>p$, let $\calV$ be an affine subspace of $\Mat_{n,p}(\K)$ and set $r:=\max\{\rk A \mid A \in \calF\}$.
Then $\dim \calV \leq nr$. \\
If in addition $\dim \calV=nr$, then
$\calV$ is equivalent to $\Bigl\{\begin{bmatrix}
M & 0
\end{bmatrix} \mid M \in \Mat_{n,r}(\K)\Bigr\}$.
\end{cor}

\begin{proof}
We embed $\calV$ into  $\Mat_n(\K)$ as an affine subspace $\calV'$ by mapping any $M\in V$ to $\begin{bmatrix}
M & 0
\end{bmatrix}$. Then Theorem \ref{Flanders} applies to $\calV'$
(notice that the exceptional case $n=2$ and $r=1$ does not occur), hence $\dim \calV=\dim \calV' \leq n\,r$. \\
Assume now $\dim \calV'=n\,r$. Notice that $\calV'$ cannot be equivalent to ${}^tC_r$ (the intersection of kernels
of matrices of $\calV'$ would be $\{0\}$, which is not the case). Hence $\calV'$ is equivalent to $C_r$,
which shows there is a $(n-r)$-dimensional subspace $F$ of $\K^n$ on which every $M' \in \calV'$ vanishes.
Then every $M \in \calV$ vanishes on $G:=F \cap (\K^p\times \{0\})$, with $\dim G\geq p-r$.
Since $\dim \calV=n\,r$, we deduce that $\dim G=p-r$ and $\calV$ is the set of matrices of $\Mat_{n,p}(\K)$ vanishing on $G$.
Hence $\calV$ is clearly equivalent to the linear subspace $\Bigl\{\begin{bmatrix}
M & 0
\end{bmatrix} \mid M \in \Mat_{n,r}(\K)\Bigr\}$ of $\Mat_{n,p}(\K)$.
\end{proof}

\section{The case of $\Mat_2(\F_2)$}

\subsection{A quick review of quadratic forms in characteristic $2$}

Let $\K$ be a field of characteristic $2$ and $V$ an $n$-dimensional vector space over $\K$. \\
A quadratic form on $V$ is a map of the type $q : x \mapsto b(x,x)$ for some
(\emph{a priori} non-symmetric) bilinear form $b : V \times V \rightarrow \K$. To such a quadratic form $q$ is assigned
a \emph{polar form} $b_q : (x,y) \mapsto q(x+y)-q(x)-q(y)$, which is always an \emph{alternate} bilinear form.
We say that $q$ is \emph{regular} (or \emph{non-degenerate}) when its polar form is non-degenerate, i.e. symplectic
(notice this implies that $n$ is even).
\vskip 2mm
\noindent Two quadratic forms $q_1$ and $q_2$ on $V$ are called \emph{equivalent} when there exists a linear automorphism $u$ of $V$
such that $q_2=q_1\circ u$.
\vskip 2mm
\noindent Given a basis $\bfB$ of $V$, a \emph{representing matrix} for $q$ in $\bfB$ is a matrix $A\in \Mat_n(\K)$
such that $q(x)=X^t\,A\,X$ for every $x \in V$ with associated column matrix $X$ in $\bfB$.
In particular, $q$ is represented in $\bfB$ by a unique upper triangular matrix $T$, the ``alternate part" $T+T^t$ of
which represents $b_q$ in $\bfB$.
\vskip 2mm
\noindent Assume now $q$ is regular, and let $\bfB$ denote a symplectic basis for $b_q$.
Then there exist two diagonal matrices $D_1$ and $D_2$ such that
$\begin{bmatrix}
D_1 & I_n \\
0 & D_2
\end{bmatrix}$ represents $q$ in $\bfB$.
Setting $\calP(\K):=\{x^2+x \mid x \in \K\}$, which is a subgroup of $(\K,+)$, it can then be proven that the class $\Delta(q)$ of $\tr(D_1D_2)$
in the quotient group $\K/\calP(\K)$ is independent on the choice of $\bfB$: this is called the \textbf{Arf invariant} of $q$.
Two equivalent regular quadratic forms have the same Arf invariant, and the converse is true if $\K$ is perfect, in particular
when $\K$ is finite. Notice that $\calP(\K)=\{0\}$ when $\K=\F_2$, in which case the Arf invariant is naturally considered as an element of $\F_2$.

\subsection{From $\Sp_4(\F_2)$ to affine automorphisms of $\Mat_2(\F_2)$ \ldots}

We set here $V=\F_2^4$ which we equip with the canonical symplectic form $b$ (for which the canonical basis $\bfB$ is symplectic).
The set $\calQ(b)$ of all quadratic forms on $V$ with polar form $b$ is an affine subspace of $\F_2^V$
with the dual space $V^\star$ as translation vector space. Given $u \in \Sp(b)$ and $q \in \calQ(b)$,
the quadratic form $q \circ u$ has polar form $b$, hence $(q,u) \mapsto q \circ u^{-1}$
defines a left-action of $\Sp(b)$ on $\calQ(b)$.
We set
$$\psi : \Sp(b) \rightarrow \frak{S}(\calQ(b))$$
as the associated homomorphism and notice that, for every $u \in \Sp(b)$, the map
$\psi(u)$ is an affine transformation which preserves the Arf invariant.

\noindent Let $q \in \calQ(b)$. Since $\bfB$ is a symplectic basis, there is a unique
$M(q)=\begin{bmatrix}
a & d \\
b & c
\end{bmatrix}$ such that $q$ is represented in $\bfB$ by the upper triangular matrix
$$\begin{bmatrix}
a & 0 & 1 & 0 \\
0 & b & 0 & 1 \\
0 & 0 & c & 0 \\
0 & 0 & 0 & d
\end{bmatrix}.$$
Notice that $q \mapsto M(q)$ is an affine isomorphism from $\calQ(b)$ to $\Mat_2(\F_2)$
and $\forall q \in \calQ(b), \; \det M(q)=\Delta(q)$. In particular, the set $\calQ_1(b)$ of elements $q \in \calQ(b)$ with $\Delta(q)=1$
has six elements.

\begin{prop}\label{premieriso}
The homomorphism $\overline{\psi} : \Sp(b) \mapsto \frak{S}(\calQ_1(b))$ induced by $\psi$ is an isomorphism.
\end{prop}

\begin{Rem}
It can easily deduced from there that $O(q) \simeq \frak{S}_5$ for any $q \in \calQ_1(b)$.
\end{Rem}

\begin{proof}
Notice that both groups $\Sp(b)$ and $\frak{S}(\calQ_1(b))$ have order $6!$ (see page 147 paragraph III.6 of \cite{Artin}).
It will then suffice to show that $\overline{\psi}$ is one-to-one.
Let $u \in \Sp(b) \setminus \{\id_V\}$. We will exhibit a quadratic form
$q$ in $\calQ_1(b)$ for which $u$ is not an orthogonal automorphism, so that $\psi(u)[q] \neq q$.
We choose a non-zero $x \in V$ such that $u(x) \neq x$, hence $x$ and $u(x)$ are not colinear.
We then exhibit a quadratic form $q \in \calQ_1(b)$ such that $q(x)=0$ and $q(u(x))=1$.
\begin{itemize}
\item If $b(x,u(x))=0$, then we extend $(x,u(x))$ into a symplectic basis $\bfB$ of $V$, so
$\begin{bmatrix}
0 & 0 & 1 & 0 \\
0 & 1 & 0 & 1 \\
0 & 0 & 0 & 0 \\
0 & 0 & 0 & 1
\end{bmatrix}$ represents in $\bfB$ a quadratic form that suits our needs.
\item If $b(x,u(x))=1$, then we choose a symplectic basis $(y,z)$ of $\{x,u(x)\}^\bot$, so
$\begin{bmatrix}
0 & 1 & 0 & 0 \\
0 & 1 & 0 & 0 \\
0 & 0 & 1 & 1 \\
0 & 0 & 0 & 1
\end{bmatrix}$ represents in $(x,u(x),y,z)$ a quadratic form that suits our needs.
\end{itemize}
\end{proof}

\noindent The affine isomorphism $q \mapsto M(q)$ induces an isomorphism from the affine group of
$\calQ(b)$ to that of $\Mat_2(\F_2)$ which maps the preservers of the Arf invariant to the preservers of $\GL_2(\F_2)$.
Right-composing it with $\psi$ yields then a group homomorphism $\varphi : \Sp(b) \rightarrow \calA\calG_2(\F_2)$.
Together with the homomorphism associated with the natural action of $\calA\calG_2(\F_2)$ on $\GL_2(\F_2)$, this defines a sequence:
$$\Sp(b) \overset{\varphi}{\longrightarrow} \calA\calG_2(\F_2) \longrightarrow \frak{S}(\GL_2(\F_2)).$$

\begin{prop}\label{deuxiemeiso}
The homomorphism $\Sp(b) \rightarrow \calA\calG_2(\F_2)$ induced by $\psi$ and $M$ is an isomorphism, and so is
the natural homomorphism $\calA\calG_2(\F_2) \rightarrow \frak{S}(\GL_2(\F_2))$
\end{prop}

\begin{proof}
By Lemma \ref{GLaffengendre}, every affine automorphism of $\Mat_2(\F_2)$
which fixes every non-singular matrix must be the identity, hence the second map is one-to-one.
The conclusion follows then readily from Proposition \ref{premieriso}.
\end{proof}

\noindent We may now easily compute an explicit affine automorphism of $\Mat_2(\F_2)$
which is not linear. The six elements of $\GL_2(\F_2)$ are:
$$\begin{bmatrix}
1 & 0 \\
0 & 1
\end{bmatrix}, \; \begin{bmatrix}
0 & 1 \\
1 & 0
\end{bmatrix}, \quad \begin{bmatrix}
1 & 1 \\
0 & 1
\end{bmatrix}, \quad \begin{bmatrix}
1 & 0 \\
1 & 1
\end{bmatrix}, \quad \begin{bmatrix}
0 & 1 \\
1 & 1
\end{bmatrix} \quad \text{and} \quad \begin{bmatrix}
1 & 1 \\
1 & 0
\end{bmatrix}.$$
Define $u \in \calA\calG_2(\F_2)$ as the element which permutes the first two and fixes all the others.
Notice that the four last matrices form a basis of $\Mat_2(\F_2)$, hence $u$ cannot be linear.
More explicitly, a straightforward computation shows:
$$\forall (a,b,c,d)\in \F_2^4, \quad u\begin{bmatrix}
a & c \\
b & d
\end{bmatrix}=\begin{bmatrix}
b+c+d+1 & a+b+d+1 \\
a+c+d+1 & a+b+c+1
\end{bmatrix}.$$

\subsection{\ldots and back again to $\Sp_4(\F_2)$}

Notice now that the determinant is a regular quadratic form on $\Mat_2(\F_2)$
(in the canonical basis $(E_{1,1},E_{2,2},E_{1,2},E_{2,1})$ of $\Mat_2(\F_2)$, it is represented by the
matrix $\begin{bmatrix}
0 & 1 & 0 & 0 \\
0 & 0 & 0 & 0 \\
0 & 0 & 0 & 1 \\
0 & 0 & 0 & 0
\end{bmatrix}$, with a non-singular alternate part).
Let $B$ denote its polar form, so that $B(X,Y)=\det(X+Y)-\det X-\det Y$ for every $(X,Y)\in \Mat_2(\F_2)^2$.
Let $u \in \calA\calG_2(\F_2)$ and $\vec{u}$ be its linear part.
Then $u$ is a determinant preserver (since $u$ is bijective and stabilizes the finite set $\GL_2(\F_2)$).
For every $M \in \Mat_2(\F_2)$, equating $\det(u(M))$ with $\det M$ yields:
$$\det u(0)+B(u(0),\vec{u}(M))+\det(\vec{u}(M))=\det 0+B(u(0),M)+\det(M).$$
Taking the polar form on both sides then shows that $\vec{u}$ is a symplectic automorphism for $B$.
From that, we deduce a group homomorphism:
$$\alpha : \calA\calG_2(\F_2) \longrightarrow \Sp(B)$$
which we claim is an isomorphism. Since both groups have the same order, it will suffice to show $\alpha$ is one-to-one.
If not, then there would be a non-identity translation which preservers the determinant on $\Mat_2(\F_2)$, hence
a non-zero matrix $A$ such that $\det(A+M)=\det M$ for every $M \in \Mat_2(\F_2)$.
This would yield $\forall M \in \Mat_2(\F_2), \; \det A+B(A,M)=0$, hence $\det A=0$, and then $B(A,-)=0$ with $A \neq 0$. This would
contradict the fact that $B$ is symplectic. We conclude that $\alpha$ is an isomorphism, which was the final statement in Theorem \ref{affpresdSP}.

\section*{Acknowledgements}
This article is dedicated to Saab Abou-Jaoud\'e, who has, some time ago, offered me crucial help in my research
on the Flanders' theorem.


\begin{thebibliography}{1}
\bibitem{Artin}
E. Artin,
\newblock{\em Geometric Algebra,}
\newblock{Interscience Tracts in Pure and Applied Mathematics, 3,}
\newblock{Wiley Classics Library, Interscience Publishers, New-York, 1988.}

\bibitem{Dieudonne}
J. Dieudonn\'e, Sur une g\'en\'eralisation du groupe orthogonal \`a quatre variables,
\newblock{\em Arch. Math,}
\newblock{\textbf{1}}
\newblock{(1949),}
\newblock{282-287.}

\bibitem{Flanders}
H. Flanders, On spaces of linear transformations with bounded rank,
\newblock{\em J. Lond. Math. Soc,}
\newblock{\textbf{37}}
\newblock{(1962),}
\newblock{10-16.}

\bibitem{Frobenius}
G. Frobenius, Uber die Darstellung der endlichen Gruppen durch Lineare Substitutionen,
\newblock{\em Sitzungsber
Deutsch. Akad. Wiss. Berlin,}
\newblock{(1897),}
\newblock{994-1015.}

\bibitem{Greub}
W. Greub,
\newblock{\em Linear algebra, 3rd edition,}
\newblock{Grundlehren der mathematiken Wissenschaften, 97}
\newblock{Springer-Verlag, 1967.}

\bibitem{Meshulam}
R. Meshulam, On the maximal rank in a subspace of matrices,
\newblock{\em  Q. J. Math., Oxf. II,}
\newblock{\textbf{36}}
\newblock{(1985),}
\newblock{225-229.}

\bibitem{Scharlau}
W. Scharlau,
\newblock{\em Quadratic and Hermitian Forms,}
\newblock GmW, 270, Springer-Verlag, 1985.

\bibitem{dSPlinpresGL}
C. de Seguins Pazzis, The singular linear preservers of non-singular matrices,
\newblock{\em  Lin. Alg. Appl,}
\newblock{accepted, doi:10.1016/j.laa.2010.03.021}


\end{thebibliography}
\end{document}